\documentclass[12pt]{article}
\usepackage{amsmath,amsthm,amsfonts,amssymb,longtable,enumerate}
\usepackage{hyperref,url}

\newtheorem{thm}{Theorem}[section]
\newtheorem{lem}[thm]{Lemma}
\newtheorem{prop}[thm]{Proposition}
\theoremstyle{definition}
\newtheorem{dfn}[thm]{Definition}

\newtheorem{rem}[thm]{Remark}
\newtheorem{exam}[thm]{Example}

\DeclareMathOperator{\PG}{PG}
\DeclareMathOperator{\PSL}{PSL}
\DeclareMathOperator{\PGL}{PGL}
\newcommand{\F}{\mathbb{F}}

\begin{document}
\title{$3$-Designs from $\PSL(2,q)$ with cyclic starter blocks}

\author{
  Akihide Hanaki\footnote{
    A. Hanaki : Faculty of Science,
    Shinshu University, Matsumoto 390-8621, Japan, 
    e-mail address : hanaki@shinshu-u.ac.jp}
  \and Kenji Kobayashi\footnote{
    K. Kobayashi : Graduate School of Science and Technology,
    Shinshu University, Matsumoto 390-8621, Japan,
    e-mail address : 23ss106c@shinshu-u.ac.jp}
  \and Akihiro Munemasa\footnote{
    A. Munemasa : Graduate School of Information Sciences,
    Tohoku University, Sendai 980-8579, Japan
    e-mail address : munemasa@tohoku.ac.jp}
}
\date{}
\maketitle

\begin{abstract}
  We consider when the projective special linear group
  over a finite field defines a 
  block-transitive
  $3$-design
  with a starter block
  which is a multiplicative subgroup of the field.
  For a prime power $q\equiv 1\pmod{20}$,
  we will show that the multiplicative subgroup of order $5$
  is a starter block of a $3$-design if and only if
  the multiplicative subgroup of order $10$
  is a starter block of a $3$-design.
  The former is the family of
  $3$-$(q+1,5,3)$ investigated by Li, Deng and Zhang,
  while the latter appear in a different context by
  Bonnecaze and Sol\'e for the case $q=41$.
  We also show a similar equivalence for multiplicative subgroups
  of orders $13$ and $26$ 
  for a prime power $q\equiv 1\pmod{52}$.
\end{abstract}

\noindent
keywords: $3$-design, projective special linear group

\section{Introduction}
Bonnecaze and Sol\'e \cite{MR4275361} found that the extended quadratic residue
code of length $42$ supports a (seemingly sporadic) $3$-$(42,10,18)$ design. 
It turns out that this design has $\mathrm{PSL}(2,41)$ as a flag-transitive
automorphism group, and has the multiplicative subgroup of $10$-th roots of 
unity in $\mathbb{F}_{41}$ as a starter block.

The first purpose of this article is to show that this $3$-$(42,10,18)$ design is 
the first one in the family of flag-transitive $3$-$(q+1,10,18)$ designs,
where $q$ is an odd power of a prime in the sequence A325072 in OEIS \cite{OEIS}:
\begin{equation}\label{OEIS_A325072}
41,61,241,281,421,601,641,\dots.
\end{equation}

It is well known that $\mathrm{PSL}(2,q)$ acts $3$-transitively on $q+1$ points
if $q\not\equiv1\pmod{4}$. For such a prime power $q$, an arbitrary $k$-subset
with $k\geq3$ is a starter block of a $3$-design. The problem is then to determine
the parameter $\lambda$ of the design, which can be reduced to finding the
order of the stabilizer of the starter. See
\cite{MR2273135,MR2036648}.

However, for $q\equiv1\pmod{4}$, the action of $\mathrm{PSL}(2,q)$ on $q+1$
points has two orbits on triples, so there is no guarantee that a $k$-subset
can be a starter block of a $3$-design. In fact, the only such instance,
as far as the authors are aware of, is the case $k=5$ by
Li, Deng and Zhang \cite{MR3752590}, in addition to the aforementioned
paper \cite{MR4275361}. Since Bonnecaze and Sol\'e \cite{MR4275361} did not
cite \cite{MR3752590}, they seem to be unaware of the connection between
\cite{MR4275361} and \cite{MR3752590} which we aim to reveal in the present paper.

The sequence \eqref{OEIS_A325072} consists of primes $p$ satisfying
$p\equiv1\pmod{20}$ and one of the following equivalent
conditions:
\begin{enumerate}[(i)]
  \item there exists $\theta\in\mathbb{F}_p^\times
  \setminus(\mathbb{F}_p^\times)^2$ such that 
  $\theta^2-4\theta-1=0$,
  \item $p\neq x^2+20y^2$ for any integers $x,y$,
  \item $p\neq x^2+100y^2$ for any integers $x,y$,
  \item $5$ is a not a quartic residue in $\mathbb{F}_p$.
\end{enumerate}

Li, Deng and Zhang \cite{MR3752590}, show that, if $p$ satisfies
the above condition (i), then 
the orbit of $\{1,\beta,\beta^2,\beta^3,\beta^4\}$, where $\beta$ is a primitive $5$-th root of unity,
under $\mathrm{PSL}(2,p)$ is a flag-transitive $3$-$(p+1,5,3)$ design.
Moreover, they showed that $p$ can be a prime power, not necessarily a prime,
as long as condition (i) is satisfied.
We will show that,
for the same prime power $p$ satisfying the condition (i),
if $q$ is an odd power of $p$, then 
the orbit of $\{1,\gamma,\gamma^2,\dots,\gamma^9\}$,
where $\gamma$ is a primitive $10$-th root of unity, under
$\mathrm{PSL}(2,q)$ is a flag-transitive $3$-$(q+1,10,18)$ design
(Proposition \ref{prop2.4} and Theorem \ref{thm5-10}).

In order to describe the results in this paper, we introduce some terminology.
Let $q$ be an odd prime power
with $q\equiv1\pmod4$, 
and consider $G=\PSL(2,q)$ acting on $\PG(1,q)=\F_q\cup\{\infty\}$.
For a divisor $k$ of $q-1$ with $3<k<q-1$, let $B$ be the unique subgroup of
order $k$ in the multiplicative group of $\F_q$.
We say that \emph{$(q,k)$ gives a $3$-design}
if $GB$ is a $3$-design.
We will show that, if $(q,k)$ gives a $3$-design,
then $k\equiv 1,2,5,10,13,17\pmod{24}$ (Theorem \ref{lem2.1}).
The result mentioned in the previous paragraph is also formulated
as Theorem~\ref{thm5-10}: 
$(q,5)$ gives a $3$-design if and only if $(q,10)$ gives a $3$-design for $q\equiv 1\pmod{20}$.
In Theorem \ref{thm13-26},
we show that $(q,13)$ gives a $3$-design if and only if $(q,26)$ gives a $3$-design for $q\equiv 1\pmod{52}$.
By computer search, we could not find any more such pair of $k$. 

\section{Preliminaries}

Let $X$ be a finite set, and let $G$ be a finite permutation group on $X$.
The group $G$ is said to be $t$-homogeneous if $G$ is transitive on $\binom{X}{t}$,
where $\binom{X}{t}$ is the set of all $t$-subsets of $X$.
Obviously a $t$-transitive group is $t$-homogeneous.
Fix $B\in\binom{X}{k}$ for $k\geq t$, and set $GB=\{gB\mid g\in G\}$.
It is clear that, if $G$ is $t$-homogeneous on $X$, then
$GB$ is a block-transitive combinatorial $t$-design on $X$.

Let $q$ be a prime power.
Consider the projective special linear group $\PSL(2,q)$ over a finite field $\F_q$ of order $q$
acting on $\PG(1,q)=\F_q\cup\{\infty\}$ by
\begin{equation}\label{eq:LF}
  \begin{bmatrix}a&b\\c&d\end{bmatrix}z=\frac{az+b}{cz+d},\quad
  \text{where }z\in\PG(1,q),
  \;\begin{bmatrix}a&b\\c&d\end{bmatrix}
  \in\mathrm{PSL}(2,q).
\end{equation}
We will focus on the case where 
$q$ is odd.
It is known that the action is $2$-transitive,
$3$-homogeneous if $q\equiv 3\pmod{4}$
and, if $q\equiv 1\pmod{4}$, 
then  there are exactly two $\mathrm{PSL}(2,q)$-orbits $\mathcal{O}_+$ and $\mathcal{O}_-$
on $\binom{\PG(1,q)}{3}$
with representatives $T_+=\{\infty,0,1\}$ and $T_-=\{\infty,0,\alpha\}$, 
respectively, where $\alpha$ is a primitive element of $\mathbb{F}_q$.
In the latter case, 
by Tonchev \cite{MR940701},
for a $k$-subset $B$ of $X$, the $\PSL(2,q)$-orbit containing  $B$ forms a $3$-design
if and only if
\[\left|\mathcal{O}_{+}\cap\binom{B}{3}\right|=\left|\mathcal{O}_{-}\cap\binom{B}{3}\right|.\]

The following is a special case of \cite[Theorem~1.6.1]{MR940701}.
See also Li, Deng and Zhang \cite[Lemma 2.8]{MR3752590}.

\begin{lem}\label{lem:3}
  Let $G$ be a permutation group on $X$, and suppose that there are
  exactly two $G$-orbits $\mathcal{O}_{\pm}$  on $\binom{X}{3}$.
  For a $k$-subset $B$ of $X$, the $G$-orbit of $B$ forms a $3$-design
  if and only if
  \[\left|\mathcal{O}_{+}\cap\binom{B}{3}\right|=\left|\mathcal{O}_{-}\cap\binom{B}{3}\right|.\]
\end{lem}

Recall that $\PSL(2,q)$ acts on the projective line
$\PG(1,q)=\F_q\cup\{\infty\}$
via linear fractional transformations \eqref{eq:LF}.

\begin{lem}[{\cite{MR1997556}}]\label{lem:1}
  Let $q$ be a prime power with $q\equiv1\pmod{4}$. 
  Then there are exactly two $\PSL(2,q)$-orbits $\mathcal{O}_+$ and $\mathcal{O}_-$
  on $\PG(1,q)$,
  with representatives $T_+=\{\infty,0,1\}$ and $T_-=\{\infty,0,\alpha\}$, 
  respectively, where $\alpha$ is a primitive element of $\F_q$.
  Moreover, the setwise stabilizer of $T_\pm$ in $\PSL(2,q)$ acts as
  the symmetric group of degree $3$ on $T_\pm$.
\end{lem}

\begin{lem}[{\cite[Lemma 4.3]{MR1075708}}]\label{lem:2}
  Under the same assumptions and notation as in Lemma~\ref{lem:1}, 
  let 
  $\chi\colon\F_q^\times\to\{\pm1\}$ denote the quadratic residue character,
  namely, $\chi(a)=1$ if $a\in (\F_q^\times)^2=\{b^2\mid b\in \F_q^\times\}$,
  and $\chi(a)=-1$ otherwise.
  Define
  $$
  \Delta:\binom{\F_q}{3}\to\{\pm1\},\quad
  \Delta(\{z_1,z_2,z_3\})=
  \chi((z_1-z_2)(z_2-z_3)(z_3-z_1)).$$
  Then
  \[\binom{\F_q}{3}\cap\mathcal{O}_{\pm}=\Delta^{-1}(\pm1).\]
\end{lem}

\begin{proof}
  Observe that $\chi(-1)=1$ since $q\equiv1\pmod{4}$.
  This implies that $\Delta$ is well-defined. 
  If $w=f(z)$ is a linear fractional transformation as in 
  \eqref{eq:LF}, then
  \begin{align*}
    &(f(z_1)-f(z_2))(f(z_2)-f(z_3))(f(z_3)-f(z_1))
    \\&=\left((cz_1+d)(cz_2+d)(cz_3+d)\right)^2
    (ad-bc)^3(z_1-z_2)(z_2-z_3)(z_3-z_1).
  \end{align*}
  Thus $\Delta(\{z_1,z_2,z_3\})=\Delta(\{f(z_1),f(z_2),f(z_3)\})$,
  in other words, $\Delta$ is constant on each of
  $\binom{\F_q}{3}\cap\mathcal{O}_\pm$. The proof will be complete
  if we show that $\Delta$ is surjective. This is immediate
  since $\Delta(\{\alpha z_1,\alpha z_2,\alpha z_3\})=-\Delta(\{z_1,z_2,z_3\})$ for
  $\{z_1,z_2,z_3\}\in\binom{\F_q}{3}$.
\end{proof}

We will consider the following setting.

\begin{dfn}\label{dfn:starter}
  Let $q$ be a prime power, and let $B$ be a subset of $\PG(1,q)$ consisting of $k$ elements.
  We say that $B$ is a \emph{starter} of a $3$-design under the action of $\PSL(2,q)$
  if the $\PSL(2,q)$-orbit containing $B$ is the set of blocks of a
  $3$-$(q+1,k,\lambda)$ design for some positive integer $\lambda$.
  Suppose
  $q\equiv 1\pmod{4}$, 
  $k$ is a divisor of $q-1$, and $e=(q-1)/k$.
  Let $\alpha$ be a primitive root of the finite field $\F_q$,
  and set $\beta=\alpha^e$, $B=\langle \beta\rangle$.
  We say that $(q,k)$ \emph{gives a $3$-design} if
  $B$ is a starter of a $3$-design under the action of $\PSL(2,q)$.
\end{dfn}

We first determine the value $\lambda$ in term of $q$ and $k$.

\begin{thm}\label{thm:1}
  Under the same notation as in Definition~\ref{dfn:starter}, suppose that $k< p$, where
  $p$ is the characteristic of the field $\F_q$.
  Then the following statements hold:
  \begin{enumerate}[{\rm (i)}]
    \item If $e$ is odd, then $(q,k)$ gives a $3$-$(q+1,k,\lambda)$ design,
    where $\lambda=\frac{1}{2}(k-1)(k-2)$.
    \item If $e$ is even and $(q,k)$ gives a $3$-$(q+1,k,\lambda)$ design,
    then $\lambda=\frac{1}{4}(k-1)(k-2)$.
  \end{enumerate}
\end{thm}

\begin{proof}
  (i). Suppose that $e$ is odd.
  By Lemma \ref{lem:2}, the multiplication by
  $\alpha^e$ interchanges $\binom{B}{3}\cap\mathcal{O}_+$ and
  $\binom{B}{3}\cap\mathcal{O}_-$. 
  By Lemma \ref{lem:3}, $(q,k)$ gives a $3$-design.
  (The assumption $k<p$ is not necessarily for this part.)
  
  To compute $\lambda$, we determine the stabilizer $S$ of $B$ in $\PSL(2,q)$.
  An element of $\PSL(2,q)$ will be denoted by $\begin{bmatrix}a&b\\c&d\end{bmatrix}$,
  where $a,b,c,d\in \F_q$ such that  $ad-bc\in (\F_q^\times)^2$.
  We remark that no non-trivial element of $\PSL(2,q)$ fixes three points in $\PG(1,q)$
  \cite[II.8.1 Hilfssatz]{MR224703}.
  Let $z\in S$ be of order $p$.
  Then the type of $z$ as a permutation on $\PG(1,q)$ must be $(1,p,\dots,p)$.
  Since we are assuming $k<p$, any $k$-subset can not be fixed by $z$.
  The prime $p$ does not divide $|S|$.
  Set
  $$x=
  \begin{bmatrix}
    \beta^2&0\\0&1
  \end{bmatrix},\quad
  y=
  \begin{bmatrix}
    0&1\\1&0
  \end{bmatrix}.$$
  Then $\langle x,y\rangle$ is a dihedral group of order $k$ contained in $S$.
  Suppose $\langle x,y\rangle$ is a proper subgroup of $S$.
  By the classification of subgroups of $\PSL(2,q)$
  \cite{MR104735}, \cite[II 8.27 Hauptzats]{MR224703},
  $S$ contains a dihedral group $D$ of order $\ell$ larger than  $\langle x,y\rangle$,
  where $k\mid \ell$ and $\ell\mid (q-1)/2$.
  Remark that $e=(q-1)/k$ is odd, and thus $\ell/k$ is also odd.  
  Let $u$ be an element in $D$ of order $\ell/2$.
  The element $u$ has exactly two fixed points in $\PG(1,q)$
  and acts semiregularly on the other points \cite[II Section 8]{MR224703}.
  The type of $u$ is $(1,1,\ell/2,\dots,\ell/2)$.
  Only when $\ell/2=k$, $u$ can stabilize a $k$-subset, but $\ell/k$ is odd now. 
  We can conclude $S=\langle x,y\rangle$ and the order is $k$.
  Now, $v=q+1$ and
  \begin{eqnarray*}
    \lambda&=&|GB|\frac{\binom{k}{3}}{\binom{v}{3}}
               = \frac{|G|}{|S|}\frac{k(k-1)(k-2)}{v(v-1)(v-2)}\\
           &=& \frac{\frac{1}{2}(q+1)q(q-1)}{k} \frac{k(k-1)(k-2)}{(q+1)q(q-1)}
               = \frac{1}{2}(k-1)(k-2).
  \end{eqnarray*}
  
  (ii). Suppose that $e$ is even. We only compute $\lambda$.
  The argument is similar to that for (i). 
  Set
  $$x=
  \begin{bmatrix}
    \beta&0\\0&1
  \end{bmatrix},\quad
  y=
  \begin{bmatrix}
    0&1\\1&0
  \end{bmatrix}.$$
  Then $\langle x,y\rangle$ is a dihedral group of order $2k$ contained in the stabilizer $S$.
  By the similar argument as in (i), we can conclude that $S=\langle x,y\rangle$
  and $\lambda=\frac{1}{4}(k-1)(k-2)$.
\end{proof}

The case $(q-1)/e=4$ in Theorem~\ref{thm:1} has been noticed
by Keranen, Kreher and Shiue, \cite{MR1997556} which states
that if $q\equiv5$ or $13\pmod{24}$, then there is a block-transitive
$3$-$(q+1,4,3)$ design.
As for codes of $3$-$(q+1,k,\lambda)$ designs from
$\PGL(2,2^m)$, see Ding, Tang and Tonchev \cite{MR4272622}.

For the remainder of this article, we use the notation of Definition \ref{dfn:starter}
and $\chi$, $\Delta$ in Lemma \ref{lem:2},
and suppose that \emph{$e$ is even}, because $(q,k)$ always gives a $3$-design, if $e$ is odd. 
\begin{lem}\label{lem:delta}
  The pair $(q,k)$ gives a $3$-design
  if and only if
  \begin{equation}\label{sum0}
    \sum_{\{x,y,z\}\in\binom{B}{3}}\Delta(\{x,y,z\})=0.
  \end{equation}
\end{lem}

\begin{proof}
  This is immediate from Lemmas~\ref{lem:3} and \ref{lem:2}.
\end{proof}

Our problem is to consider what parameters satisfy the condition (\ref{sum0}).

The dihedral group $D_{2k}$ acts naturally on the group
$B=\{1,\beta,\beta^2,\dots,\beta^{k-1}\}$ of order $k$
and thus acts on $\binom{B}{3}$.
Since we are assuming that $q\equiv 1\pmod{4}$ and $e$ is even,
we have $\chi(-1)=\chi(\beta)=1$.
This means that the action of $D_{2k}$ on $\binom{B}{3}$ preserves the value of $\Delta$.

\begin{lem}\label{lem1.6}
  The $D_{2k}$-orbits on $\binom{B}{3}$
  are as follows.
  \begin{longtable}{|c|c|c|c|}
    \hline
    Type & representative & length & \\
    \hline
    {\rm (A)} & $\{1,\beta^i,\beta^j\}$ & $2k$ & $1\leq i < j-i < k-i-j$ \\
    {\rm (B)} & $\{1,\beta^i,\beta^{2i}\}$ & $k$ & $1\leq i< k/2$, $i\ne k/3$\\
    {\rm (C)} & $\{1,\beta^{k/3},\beta^{2k/3}\}$ & $k/3$ & only if $k\equiv0\pmod{3}$\\    
    \hline
  \end{longtable}
\end{lem}

\begin{prop}\label{prop1.7}
  The following statements hold.
  \begin{enumerate}[{\rm (i)}]
    \item If $k$ is odd, then the sequence 
    $(\chi(1-\beta),\chi(1-\beta^2),\dots,\chi(1-\beta^{(k-1)/2}))$
    determines whether $(q,k)$ gives a $3$-design or not.
    \item If $k\equiv 2\pmod{4}$, then the sequence
    $(\chi(1-\beta^2),\chi(1-\beta^4),\dots,\chi(1-\beta^{k/2-1}),\chi(2))$
    determines whether $(q,k)$ gives a $3$-design or not.
  \end{enumerate}
\end{prop}

\begin{proof}
  For any $\{\beta^i,\beta^j,\beta^\ell\}\in \binom{B}{3}$, we have
  \begin{eqnarray*}
    \Delta(\{\beta^i,\beta^j,\beta^\ell\})&=&\chi((\beta^i-\beta^j)(\beta^j-\beta^\ell)(\beta^\ell-\beta^i))\\
                                          &=&\chi(1-\beta^{j-i})\chi(1-\beta^{\ell-j})\chi(1-\beta^{i-\ell})
  \end{eqnarray*}
  and $\chi(1-\beta^m)=\chi(\beta^{k-m}-1)=\chi(1-\beta^{k-m})$.
  The statement (i) is clear.

  Suppose $k\equiv 2\pmod{4}$.
  For an odd $\ell\ne k/2$,
  $$\chi(1-\beta^{\ell})= \chi((1-\beta^{\ell})(1+\beta^{\ell})(1+\beta^{\ell}))
  =\chi(1-\beta^{2\ell})\chi(1-\beta^{{\ell}+k/2}).$$
  Remark that $2\ell$ and $\ell+k/2$ are even.
  For ${\ell}=k/2$, $\beta^{\ell}=-1$ and $\chi(1-\beta^{\ell})=\chi(2)$.
  Now (ii) holds.
\end{proof}

\begin{rem}
  We remark that the sequence in Proposition \ref{prop1.7} depends
  on the choice of the primitive root $\alpha$ (see Remark \ref{rem5.4}).
  In Theorem \ref{lem2.1}, we will show that, if $(q,k)$ gives a $3$-design,
  then $k\not\equiv 0\pmod{4}$.
\end{rem}

\section{Possible parameters}\label{sec:parameter}
In this section, we consider possibility of the parameter $(q,k)$.
By computer search, the value of $k$ is restricted.
We begin with the possibility of $k$.

\begin{thm}\label{lem2.1}
  If $(q,k)$ gives a $3$-design,
  then $k\equiv 1,2,5,10,13,17\pmod{24}$.
\end{thm}

\begin{proof}
  \noindent
  \textbf{Claim 1}.  We will show that $k\not\equiv 0\pmod{3}$.
  Suppose $k\equiv 0\pmod{3}$. The $D_{2k}$-orbit of Type (C) in Lemma \ref{lem1.6} exists.
  Then  $\sum_{\{x,y,z\}\in\binom{B}{3}}\Delta(\{x,y,z\})\equiv k/3 \pmod{k}$ cannot be zero.

  In the following, we suppose $k\not\equiv 0\pmod{3}$.
  The $D_{2k}$-orbit of Type (C) does not exist.

  \noindent
  \textbf{Claim 2}.  We will show that $k\not\equiv 0,3\pmod{4}$.
  Suppose $k\equiv 0, 3\pmod{4}$.
  Then the number of $D_{2k}$-orbits of Type (B) is odd.
  Thus $\sum_{\{x,y,z\}\in\binom{B}{3}}\Delta(\{x,y,z\})\equiv k \pmod{2k}$ cannot be zero.

  \noindent
  \textbf{Claim 3}.  We will show that $k\not\equiv 6\pmod{8}$.
  Suppose $k\equiv 6\pmod{8}$ and set $k=8\ell+6$.
  There are $(4\ell+2)$ orbits of Type (B).
  Since $\Delta(\{1,\beta^i,\beta^{2i}\})=\Delta(\{1,\beta^{k/2-i},\beta^{2(k/2-i)}\})$,
  $\sum_{\{x,y,z\}:\text{ Type (B)}}\Delta(\{x,y,z\})\equiv 2k \pmod{4k}$.
  The number of orbits of Type (A) is
  $$\frac{1}{2k}\left(\binom{k}{3}-(4\ell+2)k\right)=\frac{2(2\ell+1)(4\ell+1)}{3}$$
  and this is even.
  Hence $\sum_{\{x,y,z\}:\text{ Type (A)}}\Delta(\{x,y,z\})\equiv 0 \pmod{4k}$
  and
  $$\sum_{\{x,y,z\}\in\binom{B}{3}}\Delta(\{x,y,z\})\equiv 2k \pmod{4k}$$
  cannot be zero.

  Combining claims above, the proof is completed.
\end{proof}

\begin{exam}
  For small $k$, we  list primes $p$ such that $(p,k)$ give $3$-designs. 
  \begin{longtable}{|r|r|l|}
    \hline \endfoot 
    \hline
    $k$ & $k \mod {24}$  & $p$\\
    \hline \endhead
    $5$ & $5$ & $41,61,241,\dots$\\
    $10$ & $10$ & $41,61,241,\dots$\\
    $13$ & $13$ & $3121,3797,4993,\dots$\\
    $17$ & $17$ & $2381,4421,6529,\dots$\\
    $25$ & $1$ & $601,4001,6701,\dots$\\
    $26$ & $2$ & $3121,3797,4993,\dots$\\
    $29$ & $5$ & $6961,9049,18097,\dots$\\
    $34$ & $10$ & $613,1973,2789,\dots$\\
    $37$ & $13$ & $3257,32561,38333,\dots$\\
    $41$ & $17$ & $5413,13613,21649,\dots$\\
    $49$ & $1$ & $14897,29989,44101,\dots$\\
    $50$ & $2$ & $601,1601,4001,\dots$\\
    $53$ & $5$ & $61057, 127837, 140557,\dots$\\
    $58$ & $10$ & $1973,9049,9281,\dots$\\
    \hline
  \end{longtable}
  We can see that prime numbers for $k=5,10$ and $k=13, 26$ are same, respectively.
  It will be proved that these are completely the same
  in Section \ref{sec5-10} and Section \ref{sec13-26}.
  By computer search, we could find no more such pair of $k$. 
\end{exam}

We will give easy results on field extensions.

\begin{prop}\label{prop2.3}
  If $(q,k)$ gives a $3$-design, then so does $(q^n,k)$ for odd $n$.
\end{prop}

\begin{proof}
  A primitive $k$-th root $\beta$ in $\F_q$ is also a primitive $k$-th root in $\F_{q^n}$.
  Since $n$ is odd, the quadratic residue character of $\F_{q^n}$ restricts to 
  the quadratic residue character of $\F_{q}$.
  Now the assertion follows from Proposition~\ref{prop1.7}.
\end{proof}

\begin{rem}
  The pairs $(29^3, 13)$ and $(29^3,26)$ give $3$-designs,
  though $(29, 13)$ and $(29,26)$ do not.
\end{rem}

\begin{prop}\label{prop2.4}
  Suppose that $n$ is even and $k\mid q-1$.
  Then $(q^n,k)$ does not give a $3$-design.
\end{prop}

\begin{proof}
  As in Proposition \ref{prop2.3},
  a primitive $k$-th root $\beta$ in $\F_q$ is also a primitive $k$-th root in $\F_{q^n}$.
  Since $n$ is even, the quadratic residue character of $\F_{q^n}$ restricts to the
  identity character of $\F_q$. This implies that $\Delta(\{x,y,z\})=1$ for
  every $\{x,y,z\}\in\binom{B}{3}$.
  Now the assertion follows from Lemma~\ref{lem:delta}.
\end{proof}

\begin{rem}
  We know no example such that $(q^n,k)$ gives a $3$-design and $n$ is even.
\end{rem}

\section{Flag-transitive $3$-$(q+1,5,3)$
  and $3$-$(q+1,10,18)$ designs from $\PSL(2,q)$}\label{sec5-10}

In this section, we will prove the following theorem.
Some parts of the results were already obtained in Li, Deng and Zhang \cite{MR3752590},
but we will give proofs.
The value $\lambda$ is determined by Theorem \ref{thm:1}.

\begin{thm}\label{thm5-10}
  Let $q$ be a prime power with $q\equiv 1\pmod{20}$.
  Fix a primitive root $\alpha$ of $\F_q$, and set $\beta=\alpha^{(q-1)/5}$.
  The following statements are equivalent.
  \begin{enumerate}[{\rm (i)}]
    \item $(q,5)$ gives a $3$-design.
    \item $(q,10)$ gives a $3$-design.
    \item $\chi(1+\beta)=-1$.
    \item There exists $\theta\in \F_q^\times$
    such that $\chi(\theta)=-1$ and $\theta^2-4\theta-1=0$.
    \item $5\not\in \langle \alpha^4\rangle$.
  \end{enumerate}
  Moreover, if $q=p$ is a prime number,
  then each of the following conditions is equivalent to any of the above.
  \begin{enumerate}[{\rm (i)}]
    \setcounter{enumi}{5}
    \item There exists no $(x,y)\in \mathbb{Z}^2$
    such that $p=x^2+20y^2$.
    \item There exists no $(x,y)\in \mathbb{Z}^2$
    such that $p=x^2+100y^2$.
  \end{enumerate}
\end{thm}

To prove Theorem \ref{thm5-10}, we need some lemmas.

\begin{lem}
  The conditions (i) and (iii) in Theorem \ref{thm5-10} are equivalent. 
\end{lem}

\begin{proof}
  We use the notations in Definition~\ref{dfn:starter} with $k=5$.
  There are only two $D_{10}$-orbits on $\binom{B}{3}$.
  Representatives are $T_1=\{1,\beta,\beta^2\}$, $T_2=\{1,\beta^2,\beta^4\}$,
  and both orbits have length $5$.
  We have
  \begin{eqnarray*}
    \Delta(T_1)&=&\chi(1-\beta)\chi(\beta-\beta^2)\chi(\beta^2-1)
                   =\chi(1-\beta^2)=\chi(1-\beta)\chi(1+\beta),\\
    \Delta(T_2)&=&\chi(1-\beta^2)\chi(\beta^2-\beta^4)\chi(\beta^4-1)=\chi(1-\beta).
  \end{eqnarray*}
  Therefore, $\sum_{\{x,y,z\}\in\binom{B}{3}}\Delta(\{x,y,z\})=5\chi(1-\beta)(\chi(1+\beta)+1)$.
  This number is zero if and only if $\chi(1+\beta)=-1$.
\end{proof}

\begin{lem}
  The conditions (ii) and (iii) in Theorem \ref{thm5-10} are equivalent. 
\end{lem}

\begin{proof}
  We set $\gamma=\alpha^{(q-1)/10}$ and $B=\langle \gamma\rangle$.
  Remark that $\gamma^2=\beta$.
  There are eight $D_{20}$-orbits on $\binom{B}{3}$ as follows.
  \begin{longtable}{|c|c|c||c|c|c|}
    \hline
    representative & length & $\Delta$ & representative & length & $\Delta$ \\ \endhead
    \hline
    $T_1=\{1,\gamma,\gamma^3\}$ & $20$ & $\chi(1-\beta)$&
                                                          $T_5=\{1,\gamma,\gamma^2\}$ & $10$ & $\chi(1-\beta)$\\
    $T_2=\{1,\gamma,\gamma^4\}$ & $20$ & $\chi(1-\beta^2)$&
                                                            $T_6=\{1,\gamma^2,\gamma^4\}$ & $10$ & $\chi(1-\beta^2)$\\
    $T_3=\{1,\gamma,\gamma^5\}$ & $20$ & $\chi(2)\chi(1-\beta)$&
                                                                 $T_7=\{1,\gamma^3,\gamma^6\}$ & $10$ & $\chi(1-\beta^2)$\\
    $T_4=\{1,\gamma^2,\gamma^5\}$ & $20$ & $\chi(2)\chi(1-\beta^2)$&
                                                                     $T_8=\{1,\gamma^4,\gamma^8\}$ & $10$ & $\chi(1-\beta)$\\
    \hline
  \end{longtable}
  The values of $\Delta$ are calculated as in Proposition \ref{prop1.7}.
  Now
  $$\sum_{\{x,y,z\}\in\binom{B}{3}}\Delta(\{x,y,z\})=\chi(1-\beta)(40+20\chi(2))(1+\chi(1+\beta))$$
  is zero if and only if $\chi(1+\beta)=-1$.
\end{proof}

\begin{lem}
  The conditions (iii) and (iv) in Theorem \ref{thm5-10} are equivalent. 
\end{lem}

\begin{proof}
  Remark that $\beta^4+\beta^3+\beta^2+\beta+1=0$ and $\chi(\beta)=\chi(-1)=1$.
  Let $\theta_0=2(\beta^4+\beta)+3$ and $\theta_1=-\theta_0^{-1}$.
  Remark that $\theta_0\ne \theta_1$.
  Then $\theta_i^2-4\theta_i-1=0$ for $i=0,1$ and
  the solutions of $x^2-4x-1=0$ in $\F_q$ are only $\theta_0$ and $\theta_1$.
  Now
  $$\theta_0(1+\beta)=-\theta_0\beta^2(\beta^2+\beta+1)
  =-\beta(1+\beta)^4\in(\F_q^\times)^2,$$
  and this means $\chi(\theta_0)=\chi(1+\beta)$.
  The assertion holds by $\chi(\theta_0)=\chi(\theta_1)$.
\end{proof}

\begin{lem}
  The conditions (iii) and (v) in Theorem \ref{thm5-10} are equivalent. 
\end{lem}

\begin{proof}
  By direct calculation,
  \begin{eqnarray*}
    (\beta(1-\beta)^2(1+\beta))^2 &=& (\beta-\beta^2-\beta^3+\beta^4)^2\\
                                  &=& (-2\beta^2-2\beta^3-1)^2\\
                                  &=& 4(1+\beta+\beta^2+\beta^3+\beta^4)+5=5.
  \end{eqnarray*}
  Since $\chi( \beta(1-\beta)^2(1+\beta))=\chi(1+\beta)$,  (iii) and (v) equivalent. 
\end{proof}

The conditions (vi) and (vii) in Theorem \ref{thm5-10} are equivalent by Brink \cite{MR2473893},
and this condition is further equivalent to (v) by Hasse \cite{MR195848}.
Now the proof of Theorem \ref{thm5-10} was completed.
The prime numbers $p$ with this condition is OEIS A325072 \cite{OEIS},
noticed by Yoshinori Yamasaki.

\section{Flag-transitive $3$-$(q+1,13,33)$
  and $3$-$(q+1,26,150)$ designs from $\PSL(2,q)$}\label{sec13-26}

In this section, we will prove the following theorem.
The value $\lambda$ is determined by Theorem \ref{thm:1}.

\begin{thm}\label{thm13-26}
  Let $q$ be a prime power with $q\equiv 1\pmod{52}$.
  Fix a primitive root $\alpha$ of $\F_q$, and set $\beta=\alpha^{(q-1)/13}$.
  The following statements are equivalent.
  \begin{enumerate}[{\rm (i)}]
    \item $(q,13)$ gives a $3$-design.
    \item $(q,26)$ gives a $3$-design.
    \item  The sequence $(\chi(1-\beta),\dots,\chi(1-\beta^6))$ is one of the following:
    \begin{eqnarray*}
      \pm(1,1,-1,1,-1,-1),&&
                             \pm(1,1,-1,-1,-1,1),\\
      \pm(1,-1,1,1,-1,-1),&&
                             \pm(1,-1,1,-1,-1,1).
    \end{eqnarray*}
  \end{enumerate}
\end{thm}

The proof can be done in a manner similar to 
Section~\ref{sec5-10}, but it requires more calculations.

\begin{lem}
  The conditions (i) and (iii) in Theorem \ref{thm13-26} are equivalent. 
\end{lem}

\begin{proof}
  There are fourteen $D_{26}$-orbits on $\binom{B}{3}$.
  \begin{longtable}{|c|c|c|}
    \hline \endfoot 
    \hline
    representative & length & $\Delta$ \\ 
    \hline \endhead
    $T_1=\{1,\beta,\beta^3\}$ & $26$ & $\chi(1-\beta)\chi(1-\beta^2)\chi(1-\beta^3)$\\
    $T_2=\{1,\beta,\beta^4\}$ & $26$ & $\chi(1-\beta)\chi(1-\beta^3)\chi(1-\beta^4)$\\
    $T_3=\{1,\beta,\beta^5\}$ & $26$ & $\chi(1-\beta)\chi(1-\beta^4)\chi(1-\beta^5)$\\
    $T_4=\{1,\beta,\beta^6\}$ & $26$ & $\chi(1-\beta)\chi(1-\beta^5)\chi(1-\beta^6)$\\
    $T_5=\{1,\beta^2,\beta^5\}$ & $26$ & $\chi(1-\beta^2)\chi(1-\beta^3)\chi(1-\beta^5)$\\
    $T_6=\{1,\beta^2,\beta^6\}$ & $26$ & $\chi(1-\beta^2)\chi(1-\beta^4)\chi(1-\beta^6)$\\
    $T_7=\{1,\beta^2,\beta^7\}$ & $26$ &$\chi(1-\beta^2)\chi(1-\beta^5)\chi(1-\beta^6)$\\
    $T_8=\{1,\beta^3,\beta^7\}$ & $26$ & $\chi(1-\beta^3)\chi(1-\beta^4)\chi(1-\beta^6)$\\
    $T_9=\{1,\beta,\beta^2\}$ & $13$ & $\chi(1-\beta^2)$\\
    $T_{10}=\{1,\beta^2,\beta^4\}$ & $13$ & $\chi(1-\beta^4)$\\
    $T_{11}=\{1,\beta^3,\beta^6\}$ & $13$ & $\chi(1-\beta^6)$\\
    $T_{12}=\{1,\beta^4,\beta^8\}$ & $13$ & $\chi(1-\beta^5)$\\
    $T_{13}=\{1,\beta^5,\beta^{10}\}$ & $13$ & $\chi(1-\beta^3)$\\
    $T_{14}=\{1,\beta^6,\beta^{12}\}$ & $13$ & $\chi(1-\beta)$\\
    \hline
  \end{longtable}
  We calculate the value $\sum_{\{x,y,z\}\in\binom{B}{3}}\Delta(\{x,y,z\})$ for every
  $(\chi(1-\beta),\chi(1-\beta^2),\dots,\chi(1-\beta^6))\in \{\pm1\}^6$,
  and conclude the assertion.
\end{proof}

\begin{lem}
  The conditions (ii) and (iii) in Theorem \ref{thm13-26} are equivalent. 
\end{lem}

\begin{proof}
  We set $\gamma=\alpha^{(q-1)/26}$ and $B=\langle \gamma\rangle$.
  Remark that $\gamma^2=\beta$.
  There are fifty-six $D_{52}$-orbits on $\binom{B}{3}$ as follows.
  \begin{longtable}{|c|c|c|}
    \hline \endfoot 
    \hline
    representative & length & $\Delta$ \\ 
    \hline \endhead
    $T_1=\{1,\gamma,\gamma^3\}$ & $52$ & $\chi(1-\beta^3)\chi(1-\beta^5)\chi(1-\beta^6)$\\
    $T_2=\{1,\gamma,\gamma^4\}$ & $52$ & $\chi(1-\beta)\chi(1-\beta^2)\chi(1-\beta^3)\chi(1-\beta^5)\chi(1-\beta^6)$\\
    $T_3=\{1,\gamma,\gamma^5\}$ & $52$ & $\chi(1-\beta)\chi(1-\beta^2)\chi(1-\beta^4)\chi(1-\beta^5)\chi(1-\beta^6)$\\
    $T_4=\{1,\gamma,\gamma^6\}$ & $52$ & $\chi(1-\beta)\chi(1-\beta^3)\chi(1-\beta^4)\chi(1-\beta^5)\chi(1-\beta^6)$\\
    $T_5=\{1,\gamma,\gamma^7\}$ & $52$ & $\chi(1-\beta)$\\
    $T_6=\{1,\gamma,\gamma^8\}$ & $52$ & $\chi(1-\beta)\chi(1-\beta^3)\chi(1-\beta^4)$\\
    $T_7=\{1,\gamma,\gamma^9\}$ & $52$ & $\chi(1-\beta)\chi(1-\beta^2)\chi(1-\beta^6)$\\
    $T_8=\{1,\gamma,\gamma^{10}\}$ & $52$ & $\chi(1-\beta)\chi(1-\beta^2)\chi(1-\beta^4)\chi(1-\beta^5)\chi(1-\beta^6)$\\
    $T_9=\{1,\gamma,\gamma^{11}\}$ & $52$ & $\chi(1-\beta^2)\chi(1-\beta^5)\chi(1-\beta^6)$\\
    $T_{10}=\{1,\gamma,\gamma^{12}\}$ & $52$ & $\chi(1-\beta^2)$\\
    $T_{11}=\{1,\gamma,\gamma^{13}\}$ & $52$ & $\chi(2)\chi(1-\beta)$\\
    $T_{12}=\{1,\gamma^2,\gamma^5\}$ & $52$ & $\chi(1-\beta)\chi(1-\beta^3)\chi(1-\beta^4)$\\
    $T_{13}=\{1,\gamma^2,\gamma^6\}$ & $52$ & $\chi(1-\beta)\chi(1-\beta^2)\chi(1-\beta^3)$\\
    $T_{14}=\{1,\gamma^2,\gamma^7\}$ & $52$ & $\chi(1-\beta)\chi(1-\beta^3)\chi(1-\beta^4)\chi(1-\beta^5)\chi(1-\beta^6)$\\
    $T_{15}=\{1,\gamma^2,\gamma^8\}$ & $52$ & $\chi(1-\beta)\chi(1-\beta^3)\chi(1-\beta^4)$\\
    $T_{16}=\{1,\gamma^2,\gamma^9\}$ & $52$ & $\chi(1-\beta)\chi(1-\beta^2)\chi(1-\beta^3)\chi(1-\beta^4)\chi(1-\beta^6)$\\
    $T_{17}=\{1,\gamma^2,\gamma^{10}\}$ & $52$ & $\chi(1-\beta)\chi(1-\beta^4)\chi(1-\beta^5)$\\
    $T_{18}=\{1,\gamma^2,\gamma^{11}\}$ & $52$ & $\chi(1-\beta^4)$\\
    $T_{19}=\{1,\gamma^2,\gamma^{12}\}$ & $52$ & $\chi(1-\beta)\chi(1-\beta^5)\chi(1-\beta^6)$\\
    $T_{20}=\{1,\gamma^2,\gamma^{13}\}$ & $52$ & $\chi(2)\chi(1-\beta^2)$\\
    $T_{21}=\{1,\gamma^3,\gamma^{7}\}$ & $52$ & $\chi(1-\beta^2)\chi(1-\beta^5)\chi(1-\beta^6)$\\
    $T_{22}=\{1,\gamma^3,\gamma^{8}\}$ & $52$ & $\chi(1-\beta^3)$\\
    $T_{23}=\{1,\gamma^3,\gamma^{9}\}$ & $52$ & $\chi(1-\beta^2)\chi(1-\beta^4)\chi(1-\beta^5)$\\
    $T_{24}=\{1,\gamma^3,\gamma^{10}\}$ & $52$ & $\chi(1-\beta^6)$\\
    $T_{25}=\{1,\gamma^3,\gamma^{11}\}$ & $52$ & $\chi(1-\beta)\chi(1-\beta^2)\chi(1-\beta^3)\chi(1-\beta^4)\chi(1-\beta^5)$\\
    $T_{26}=\{1,\gamma^3,\gamma^{12}\}$ & $52$ & $\chi(1-\beta^2)\chi(1-\beta^3)\chi(1-\beta^4)\chi(1-\beta^5)\chi(1-\beta^6)$\\
    $T_{27}=\{1,\gamma^3,\gamma^{13}\}$ & $52$ & $\chi(2)\chi(1-\beta^3)$\\
    $T_{28}=\{1,\gamma^3,\gamma^{14}\}$ & $52$ & $\chi(1-\beta)\chi(1-\beta^2)\chi(1-\beta^3)\chi(1-\beta^5)\chi(1-\beta^6)$\\
    $T_{29}=\{1,\gamma^4,\gamma^9\}$ & $52$ & $\chi(1-\beta^5)$\\
    $T_{30}=\{1,\gamma^4,\gamma^{10}\}$ & $52$ & $\chi(1-\beta^2)\chi(1-\beta^3)\chi(1-\beta^5)$\\
    $T_{31}=\{1,\gamma^4,\gamma^{11}\}$ & $52$ & $\chi(1-\beta)\chi(1-\beta^3)\chi(1-\beta^6)$\\
    $T_{32}=\{1,\gamma^4,\gamma^{12}\}$ & $52$ & $\chi(1-\beta^2)\chi(1-\beta^4)\chi(1-\beta^6)$\\
    $T_{33}=\{1,\gamma^4,\gamma^{13}\}$ & $52$ & $\chi(2)\chi(1-\beta^4)$\\
    $T_{34}=\{1,\gamma^4,\gamma^{14}\}$ & $52$ & $\chi(1-\beta^2)\chi(1-\beta^5)\chi(1-\beta^6)$\\
    $T_{35}=\{1,\gamma^5,\gamma^{11}\}$ & $52$ & $\chi(1-\beta)\chi(1-\beta^2)\chi(1-\beta^3)\chi(1-\beta^4)\chi(1-\beta^5)$\\
    $T_{36}=\{1,\gamma^5,\gamma^{12}\}$ & $52$ & $\chi(1-\beta^3)\chi(1-\beta^4)\chi(1-\beta^5)$\\
    $T_{37}=\{1,\gamma^5,\gamma^{13}\}$ & $52$ & $\chi(2)\chi(1-\beta^5)$\\
    $T_{38}=\{1,\gamma^5,\gamma^{14}\}$ & $52$ & $\chi(1-\beta^2)\chi(1-\beta^5)\chi(1-\beta^6)$\\
    $T_{39}=\{1,\gamma^5,\gamma^{15}\}$ & $52$ & $\chi(1-\beta)\chi(1-\beta^2)\chi(1-\beta^4)$\\
    $T_{40}=\{1,\gamma^6,\gamma^{13}\}$ & $52$ & $\chi(2)\chi(1-\beta^6)$\\
    $T_{41}=\{1,\gamma^6,\gamma^{14}\}$ & $52$ & $\chi(1-\beta^3)\chi(1-\beta^4)\chi(1-\beta^6)$\\
    $T_{42}=\{1,\gamma^6,\gamma^{15}\}$ & $52$ & $\chi(1-\beta)\chi(1-\beta^3)\chi(1-\beta^4)$\\
    $T_{43}=\{1,\gamma^7,\gamma^{15}\}$ & $52$ & $\chi(1-\beta)\chi(1-\beta^2)\chi(1-\beta^3)\chi(1-\beta^4)\chi(1-\beta^6)$\\
    $T_{44}=\{1,\gamma^7,\gamma^{16}\}$ & $52$ & $\chi(1-\beta^2)\chi(1-\beta^3)\chi(1-\beta^4)\chi(1-\beta^5)\chi(1-\beta^6)$\\
    \hline
    $T_{45}=\{1,\gamma,\gamma^{2}\}$ & $26$ & $\chi(1-\beta)$\\
    $T_{46}=\{1,\gamma^2,\gamma^{4}\}$ & $26$ & $\chi(1-\beta^2)$\\
    $T_{47}=\{1,\gamma^3,\gamma^{6}\}$ & $26$ & $\chi(1-\beta^3)$\\
    $T_{48}=\{1,\gamma^4,\gamma^8\}$ & $26$ & $\chi(1-\beta^4)$\\
    $T_{49}=\{1,\gamma^5,\gamma^{10}\}$ & $26$ & $\chi(1-\beta^5)$\\
    $T_{50}=\{1,\gamma^6,\gamma^{12}\}$ & $26$ & $\chi(1-\beta^6)$\\
    $T_{51}=\{1,\gamma^7,\gamma^{14}\}$ & $26$ & $\chi(1-\beta^6)$\\
    $T_{52}=\{1,\gamma^8,\gamma^{16}\}$ & $26$ & $\chi(1-\beta^5)$\\
    $T_{53}=\{1,\gamma^9,\gamma^{18}\}$ & $26$ & $\chi(1-\beta^4)$\\
    $T_{54}=\{1,\gamma^{10},\gamma^{20}\}$ & $26$ & $\chi(1-\beta^3)$\\
    $T_{55}=\{1,\gamma^{11},\gamma^{22}\}$ & $26$ & $\chi(1-\beta^2)$\\
    $T_{56}=\{1,\gamma^{12},\gamma^{24}\}$ & $26$ & $\chi(1-\beta)$\\
    \hline
  \end{longtable}
  The values of $\Delta$ are calculated as in Proposition \ref{prop1.7}.
  Recall that they can be expressed by $\beta$.
  We calculate the value $\sum_{\{x,y,z\}\in\binom{B}{3}}\Delta(\{x,y,z\})$ for every
  $(\chi(1-\beta),\chi(1-\beta^2),\dots,\chi(1-\beta^6),\chi(2))\in \{\pm1\}^7$,
  and conclude the assertion.
  We remark that the value $\chi(2)$ does not affect the results.
\end{proof}

\begin{rem}\label{rem5.4}
  We remark that the sequence  $(\chi(1-\beta),\chi(1-\beta^2),\dots,\chi(1-\beta^6))$ depends
  on the choice of the primitive root $\alpha$.
  For example, when $p=3797$, $\alpha=2,2^3,2^5,2^7$ are primitive roots and the sequences are
  as follows.
  \begin{longtable}{|c|c|}
    \hline \endfoot
    \hline
    $\alpha$ & sequence\\
    \hline \endhead
    $2$ & $(1,1,-1,1,-1,-1)$\\
    $2^3$ & $(-1,-1,1,1,1,-1)$\\
    $2^5$ & $(-1,-1,1,-1,1,1)$\\
    $2^7$ & $(-1,1,-1,1,1,-1)$\\
    \hline
  \end{longtable}
\end{rem}

\subsection*{Acknowledgements}
The first author was supported by JSPS KAKENHI Grant Number JP22K03266.
This work was supported by the Research Institute for Mathematical Sciences,
an International Joint Usage/Research Center located in Kyoto University. 

\bibliographystyle{plain}

\end{document}